\documentclass[12pt]{article}

\usepackage{fullpage}
\usepackage{changepage}
\usepackage{hyperref}
\usepackage{amsfonts,amssymb,amsthm,fancybox,amsmath,mathrsfs}
\usepackage[pdftex]{graphicx}
\usepackage{tikz,tikz-cd}
\usetikzlibrary{fit}
\usepackage{listings}
\usepackage{mathtools}
\usepackage{graphicx}
\usepackage{caption}
\usepackage{enumerate}
\usepackage{listings}
\usepackage{array}
\usepackage{relsize}

\makeatletter
\newtheorem*{rep@theorem}{\rep@title}
\newcommand{\newreptheorem}[2]{
    \newenvironment{rep#1}[1]{
        \def\rep@title{#2 \ref{##1}}
        \begin{rep@theorem}
    }
    {\end{rep@theorem}}
}
\makeatother
\newreptheorem{theorem}{Theorem}
\newreptheorem{lemma}{Lemma}
\newreptheorem{corollary}{Corollary}

\usepackage{commands}

\counterwithin{figure}{section}

\AtEndDocument{
    \par
    \medskip
    \noindent
    \begin{tabular}{@{}l@{}}
        \textsc{Department of Pure Mathematics, University of Waterloo, Waterloo, Ontario}\\
        \textit{E-mail address}: \texttt{ja2999la@uwaterloo.ca}
\end{tabular}}

\title{Isomorphism Spectra and Computably Composite Structures}
\author{Joey~Lakerdas-Gayle\footnote{I am grateful to my supervisor Barbara Csima for her helpful comments and guidance. This research was partially supported by the OGS and QEII-GSST grants.}\\
	\small Department of Pure Mathematics, University of Waterloo\\
	\small Waterloo, Ontario, Canada\\
	\small\tt ja2999la@uwaterloo.ca}
\date{\today}

\begin{document}
\maketitle

\begin{abstract}
Adapting a result of Bazhenov, Kalimullin, and Yamaleev, we show that if a Turing degree $\bfd$ is the degree of categoricity of a computable structure $\M$ and is not the strong degree of categoricity of any computable structure, then $\M$ has a pair of computable copies whose isomorphism spectrum is not finitely generated. Motivated by this result, we introduce a class of computable structures called \emph{computably composite structures} with the property that the isomorphisms between arbitrary computable copies of these structures are exactly the unions of isomorphisms between the computable copies of their components. We use this to show that any computable union of isomorphism spectra is also an isomorphism spectrum. In particular, this gives examples of isomorphism spectra that are not finitely generated.
\end{abstract}

\section{Introduction}
Mathematical structures are usually considered up to isomorphism. However, isomorphic structures may have different computability-theoretic properties. In particular, a pair of isomorphic computable structures need not have any computable isomorphisms. A major topic in computable structure theory is the complexity of isomorphisms between isomorphic computable structures. See \cite{Soare} for background in computability theory, and \cite{AshKnight} and \cite{Mon21} for background in computable structure theory.

If $\A$ and $\B$ are computable structures, we write $f:\A\cong\B$ to mean that $f$ is an isomorphism from $\A$ to $\B$. The \emph{isomorphism spectrum} of $\A$ and $\B$ is the upward-closed set of Turing degrees that compute some isomorphism $f:\A\cong\B$,
\[\IsoSpec(\A,\B)=\{\bfd:\exists f:\A\cong\B[f\leq_T\bfd]\}.\]
$\A$ is \emph{$\bfd$-computably categorical} if $\bfd\in\IsoSpec(\A,\B)$ for every computable copy $\B$ of $\A$. We say that $\A$ is \emph{computably categorical} if it is $\bfzero$-computably categorical. The \emph{categoricity spectrum of $\A$} is the upward-closed set of Turing degrees
\[\CatSpec(\A)=\{\bfd:\A\text{ is $\bfd$-computably categorical}\}=\bigcap_{\B\cong\A}\IsoSpec(\A,\B).\]
We write $\D$ for the set of all Turing degrees, and $\D_{\geq}(\bfd)$ for the set of Turing degrees that compute $\bfd$. If $\CatSpec(\A)=\D_{\geq}(\bfd)$, then $\bfd$ is the \emph{degree of categoricity of $\A$}.

The following is an example of a family of computable structures that are not computably categorical. We will use these particular structures in Section~\ref{sec:thomason}.
\begin{exmp}\label{exmp:lt}
    Given a \ce set $X\subseteq\omega$ with a fixed computable enumeration $\{x_i:i<\omega\}$, define the ordering $(\omega,<_X)$ where for all $n<m<\omega$, we have
    \[2n<_X 2m\text{ and }2x_n <_X 2n+1 <_X 2x_n+2.\]
    It is not hard to check that $(\omega,<_X)$ is a computable copy of $(\omega,<)$.
    Notice that $k\in X\Leftrightarrow(\exists n)[2k<_X n<_X 2k+2]$, so the unique isomorphism $f:(\omega,<)\cong(\omega,<_X)$ can be used to compute $X$. We can also computably build $f$ using an oracle for $X$, so $f\equiv_T X$. That is, $\IsoSpec((\omega,<),(\omega,<_X))=\D_{\geq}(\deg_T(X))$. In particular, $\IsoSpec((\omega,<),(\omega,<_{\emptyset'}))=\D_{\geq}(\bfzero')$. Moreover, if $(\omega,\prec)$ is any computable copy of $(\omega,<)$, the unique isomorphism from $(\omega,<)$ to $(\omega,\prec)$ is $\bfzero'$-computable, so $\CatSpec(\omega,<)=\D_{\geq}(\bfzero')$.
\end{exmp}

Let $\M$ be any computable structure and let $\{(\A_i,\B_i)\}_{i<\omega}$ be a list of all pairs of computable copies of $\M$. Then $\CatSpec(\M)=\bigcap_{i<\omega}\IsoSpec(\A_i,\B_i)$. For some computable structures $\M$, the categoricity spectrum is achieved by a finite intersection of isomorphism spectra. Bazhenov, Kalimullin, and Yamaleev~\cite{BKY18} defined the \emph{spectral dimension} of $\M$ as the least $k\leq\omega$ such that there is a set $X\subseteq\omega$ with $|X|=k$ and $\CatSpec(\M)=\bigcap_{i\in X}\IsoSpec(\A_i,\B_i)$. We write $\SpecDim(\M)=k$ in this case.

If a structure has degree of categoricity $\bfd$ and spectral dimension 1, we say that the structure has \emph{strong degree of categoricity} $\bfd$. The study of strong degrees of categoricity motivated the introduction of spectral dimension. In the example above, we see that $(\omega,<)$ has strong degree of categoricity $\bfzero'$.\\

Fokina, Kalimullin, and Miller~\cite{FKM10} first introduced categoricity spectra and degrees of categoricity in 2010. A structure of Miller~\cite{Mil09} (a particular algebraic field) was the first shown not to have a degree of categoricity. It has spectral dimension 1. Fokina, Frolov, and Kalimullin~\cite{FFK16} showed that for every non-zero \ce degree $\bfd$, there is a $\bfd$-computably categorical rigid structure with no degree of categoricity. These structures have infinite spectral dimension. Bazhenov, Kalimullin, and Yamaleev~\cite{BKY18}, and Csima and Stephenson~\cite{CS19} independently constructed rigid computable structures with degree of categoricity and finite spectral dimension greater than 1. These were the first examples of structures that have degrees of categoricity, but not strongly. Turetsky~\cite{Tur20} constructed the first example of a computable structure with degree of categoricity and infinite spectral dimension.

It is still unknown whether there exists a computable structure $\M$ with a degree of categoricity that is not the strong degree of categoricity of any structure. In Section~\ref{sec:never strong} we expand on a result of Bazhenov, Kalimullin, and Yamaleev~\cite{BKY20} to obtain some properties that such a structure would have:
\begin{reptheorem}{cor:never strong}[Follows from Bazhenov, Kalimullin, Yamaleev \cite{BKY16},\cite{BKY20}]
    If $\bfd$ is the degree of categoricity of a computable structure $\M$ and is not the strong degree of categoricity of any computable structure, then $\M$ has infinite spectral dimension, an infinite automorphism group, and a pair of computable copies $\A$ and $\B$ such that $\IsoSpec(\A,\B)$ is not finitely generated (that is, not equal to a finite union of cones).
\end{reptheorem}
Motivated by this, we are interested in studying isomorphism spectra that are not finitely generated. Section~\ref{sec:ccs} introduces \emph{computably composite structures}, a notion of effectively ``attaching" a collection of computable structures to the points of another computable structure. We discuss the isomorphism spectra of computably composite structures in terms of the isomorphims of their component structures. In Section~\ref{sec:H}, we use a specific class of computably composite structures to show that the class of isomorphism spectra is closed under computable unions:
\begin{reptheorem}{thm:comp unions}
    Given any two uniformly computable collections of copies $\bfA=\{\A_i:i<\omega\}$ and $\bfB=\{\B_i:i<\omega\}$ such that for each $i$, $\A_i\cong\B_i$, there exists a structure with two computable copies $\M\cong\calN$ where $\IsoSpec(\M,\calN)=\bigcup_{i<\omega}\IsoSpec(\A_i,\B_i)$.
\end{reptheorem}
In Section~\ref{sec:thomason} we apply Theorem~\ref{thm:comp unions} to a result of Thomason~\cite{Tho71} to construct a particular structure whose isomorphism spectrum is not finitely generated. Section~\ref{sec:autospec} describes some connections to work of Harizanov, Morozov, and Miller \cite{HMM10} on automorphism spectra of structures, and Section~\ref{sec:catspec} briefly considers the categoricity spectra of some computably composite structures and computes the categoricity spectrum of the particular structures from Section~\ref{sec:thomason}.
\begin{rem}
    Every structure is uniformly effectively bi-interpretable with a graph (see Chapter 6.3 of Montalb\'an's monograph \cite{Mon21}). Using techniques similar to the proof of Lemma 6.3.8 of the same reference, it can be shown that effective bi-interpretability preserves isomorphism spectra and categoricity spectra so that it is sufficient to prove Corollary~\ref{cor:never strong} and Theorem~\ref{thm:comp unions} in just the case of relational structures. Thus, in the following sections, we will assume that all structures are relational.
\end{rem}

\section{Computably composite structures}\label{sec:ccs}
We begin by introducing the notion of \emph{computably composite structures}. Informally, given a computable structure $\calS$, we wish to ``attach" a computable structure $\A_x$ to each point $x\in S$. This new structure $\calS[\A_x:x\in S]$ will be computable. Its computable copies and the isomorphisms between those copies will have desirable properties.

Throughout this paper, we use calligraphic capital letters for copies of structures, and boldface capital letters for collections of copies of structures.

\begin{defn}\label{defn: unif comp}
    Let $I$ be any computable set and let $\bfA:=\{\A_x:x\in I\}$ be a collection of computable copies of structures. We say that $\bfA$ is \emph{uniformly computable} if each $\A_x$ is a particular computable copy of a computable $\Lang_x$-structure, the languages $\{\Lang_x:x\in I\}$ are uniformly computable, and there is a computable function $f:I\to\omega$ such that $\varphi_{f(x)}=D(\A_x)$ for each $x\in I$. That is, the Turing machine with index $f(x)$ computes the atomic diagram of the structure $\mathcal{A}_x$.
\end{defn}
\begin{rem}
    In Definition~\ref{defn: unif comp}, if $x\neq y$, then $\A_x$ and $\A_y$ may not be copies of the same computable structure, and may not even have the same language or universe. In particular, the universes of these copies may not be subsets of $\omega$. However, since they are uniformly computable, there are standard uniformly computable encodings of the universes into computable subsets of $\omega$.
\end{rem}
\begin{defn}[Computably composite structures]\label{cc defn}
    Let $\calS$ be a computable $\Lang_\calS$-structure with universe $S$, and let $\bfA:=\{\A_x:x\in S\}$ be a uniformly computable collection of copies such that the universes $\{S\}\cup\{A_x:x\in S\}$ are pairwise disjoint. We define the \emph{computable composition of $\calS$ with $\bfA$} to be the structure $\calS[\bfA]$ with universe $S\cup\bigcup_{x\in S}A_x$, relational language $\Lang:=\{\mu\}\cup\Lang_\calS\cup\bigcup_{x\in S}\Lang_x$, and satisfying:
    \begin{enumerate}
        \item $\mu^{\calS[\bfA]}=\left\{(a,x):x\in S,a\in A_x\right\}\cup\{(x,x):x\in S\}$.
        \item If $R$ is an $n$-ary relation in $\Lang_\calS$ and $x_0,\dots,x_{n-1}\in S$, then $R^{\calS[\bfA]}(x_0,\dots,x_{n-1})$ iff $R^\calS(x_0,\dots,x_{n-1})$.
        \item If $Q$ is an $n$-ary relation in $\Lang_x$ and $a_0,\dots,a_{n-1}\in A_x$, then $Q^{\calS[\bfA]}(a_0,\dots,a_{n-1})$ iff $Q^{\A_x}(a_0,\dots,a_{n-1})$.
        \item No other relations hold.
    \end{enumerate}
    A structure of the form $\calS[\bfA]$ is called a \emph{computably composite structure}. The structure $\calS$ is the \emph{base structure} of $\calS[\bfA]$, and the structures $\{\A_x:x\in S\}$ are the \emph{component structures}.
\end{defn}
\begin{rem}
    In a computably composite structure, $\mu$ is a directed edge. Every element has exactly one outward $\mu$-edge, so we may use the notation $\mu(z)$ to denote the unique element for which $(z,\mu(z))\in\mu$. This is computable since the element is unique and the universe of the structure is computable. Also notice that $\mu(z)\in S$ for all elements $z$ of $\calS[\bfA]$. We may think of $\mu$ as connecting members of $\A_x$ to the point $x\in\calS$.
\end{rem}
\begin{rem}
    In $\calS$ and in each component structure, the relations are uniformly computable. By the uniform computability of $\bfA$, the relations of $\calS[\bfA]$ are also uniformly computable. Since $\{S\}\cup\{A_x:x\in S\}$ are pairwise disjoint and uniformly computable, the relation $\mu^{\calS[\bfA]}$ is computable. Thus, $\calS[\bfA]$ is a computable structure.
\end{rem}
\begin{rem}
    $\calS$ is a reduct of the substructure of $\calS[\bfA]$ with universe $S$ and for each $x\in S$, $\A_x$ is a reduct of the substructure of $\calS[\bfA]$ with universe $A_x$. As these reducts remove only trivial relations, we will identify them with the actual substructures.
\end{rem}
\begin{exmp}
    We consider a simple illustrative example where $\calS=(\{0,1,2\},E)$ is the finite directed graph where $E=\{(0,1),(1,0),(0,2),(1,2)\}$.
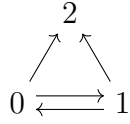
    \begin{center}
    \begin{tikzpicture}[align=center]
        \node (0) at (-0.7,0) {$0$};
        \node (1) at (0.7,0) {$1$};
        \node (2) at (0,1.2) {$2$};
        \path [->] (0.20) edge (1.160);
        \path [->] (1.200) edge (0.340);
        \path [->] (0) edge (2);
        \path [->] (1) edge (2);
    \end{tikzpicture}
    \end{center}
    Now take $\bfA:=\{\A_0,\A_1,\A_2\}$ where
    \[\A_0:=(\{0\}\times\omega,<),\ \A_1:=(\{1\}\times\omega,<),\ \A_2:=(\{2\}\times\Z,<)\]
    where $<$ is interpreted in all structures as the usual ordering in the second coordinate. We have relabeled $\omega$, $\omega$, and $\Z$ to $\{0\}\times\omega$, $\{1\}\times\omega$ and $\{2\}\times\Z$ respectively so that the universes are disjoint. Then the computably composite structure $\calS[\bfA]$ has universe
    \[\{0,1,2\}\cup\Big(\{0\}\times\omega\Big)\cup\Big(\{1\}\times\omega\Big)\cup\Big(\{2\}\times\Z\Big),\]
    and language $(\mu,E,<)$.
    \begin{figure}[!ht]
    \centering
    \begin{tikzpicture}[align=center]
        \node (0) at (-0.7,0) {$0$};
        \node (1) at (0.7,0) {$1$};
        \node (2) at (0,1.2) {$2$};
        \node (00) at (-7,-1.5) {$(0,0)$};
        \node (01) at (-5,-1.5) {$(0,1)$};
        \node (02) at (-3,-1.5) {$(0,2)$};
        \node (0dots) at (-1,-1.5) {$\dots$};
        \node (0lt1) at (-6,-1.5) {$<$};
        \node (0lt2) at (-4,-1.5) {$<$};
        \node (0lt3) at (-2,-1.5) {$<$};
        \node (10) at (1,-1.5) {$(1,0)$};
        \node (11) at (3,-1.5) {$(1,1)$};
        \node (12) at (5,-1.5) {$(1,2)$};
        \node (1dots) at (7,-1.5) {$\dots$};
        \node (1lt1) at (2,-1.5) {$<$};
        \node (1lt2) at (4,-1.5) {$<$};
        \node (1lt3) at (6,-1.5) {$<$};
        \node (21) at (2,2.7) {$(2,1)$};
        \node (22) at (4,2.7) {$(2,2)$};
        \node (2dotsup) at (6,2.7) {$\dots$};
        \node (20) at (0,2.7) {$(2,0)$};
        \node (2-1) at (-2,2.7) {$(2,-1)$};
        \node (2-2) at (-4,2.7) {$(2,-2)$};
        \node (2dotsdown) at (-6,2.7) {$\dots$};
        \node (2lt-2) at (-5,2.7) {$<$};
        \node (2lt-1) at (-3,2.7) {$<$};
        \node (2lt0) at (-1,2.7) {$<$};
        \node (2lt1) at (1,2.7) {$<$};
        \node (2lt2) at (3,2.7) {$<$};
        \node (2lt3) at (5,2.7) {$<$};
        \node[draw,dotted,fit=(00) (0dots)] {};
        \node[draw,dotted,fit=(10) (1dots)] {};
        \node[draw,dotted,fit=(2dotsup) (2dotsdown)] {};
        \path [->] (0.20) edge (1.160);
        \path [->] (1.200) edge (0.340);
        \path [->] (0) edge (2);
        \path [->] (1) edge (2);
        \path [->, red] (2dotsdown.270) edge (2);
        \path [->, red] (2-2.270) edge (2);
        \path [->, red] (2-1.270) edge (2);
        \path [->, red] (20.270) edge (2);
        \path [->, red] (21.270) edge (2);
        \path [->, red] (22.270) edge (2);
        \path [->, red] (2dotsup.270) edge (2);
        \path [->, red] (10.90) edge (1);
        \path [->, red] (11.90) edge (1);
        \path [->, red] (12.90) edge (1);
        \path [->, red] (1dots.90) edge (1);
        \path [->, red] (00.90) edge (0);
        \path [->, red] (01.90) edge (0);
        \path [->, red] (02.90) edge (0);
        \path [->, red] (0dots.90) edge (0);
    \end{tikzpicture}
    \caption{The black arrows represent the directed edge relation $E$ and the red arrows represent the directed edge relation $\mu$. The dotted boxes indicate the subsets that generate substructures $\A_0$, $\A_1$, and $\A_2$. The $\mu$-self-loops on $0,1,2$ are omitted.}
    \label{img:exmp}
    \end{figure}
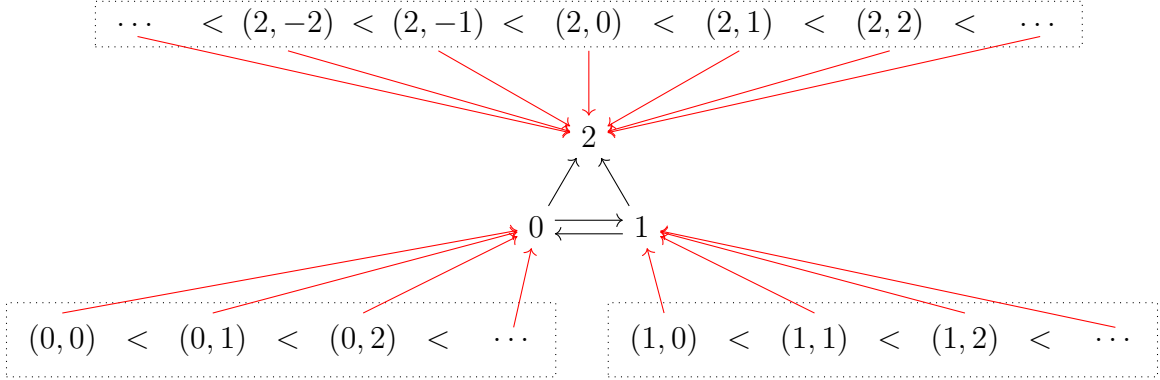
    
    We briefly consider the automorphisms of this structure. Every automorphism of $\calS[\bfA]$ fixes $2$, and may or may not swap $0$ and $1$. If an automorphism swaps $0$ and $1$, it also must swap each $(0,n)$ and $(1,n)$. An automorphism may shift each $(2,n)$ to $(2,n+k)$ for some $k\in\Z$. The automorphism is determined after choosing whether or not to swap $0$ with $1$ and choosing a $k\in\Z$, so automorphisms of $\calS[\bfA]$ are exactly the maps of the form
    \[\rho=\theta\cup\psi_0\cup\psi_1\cup\psi_2\]
    where $\theta\in\Aut(\calS)$, and $\psi_i:\A_i\cong\A_{\theta(i)}$ for $i\in\{0,1,2\}$. We will show in the next section that this characterization extends to isomorphisms between arbitrary pairs of computably composite structures.
\end{exmp}
\subsection{Isomorphisms between computably composite structures}
First, we will show that every computable copy of a computably composite structure is also computably composite.
\begin{prop}\label{copies are composite}
    Suppose $\M$ is a computable copy of $\calS[\bfA]$ via isomorphism $\rho:\calS[\bfA]\cong\M$. Let $\G$ be the substructure of $\M$ generated by $\rho(S)$. For each $g\in G$ let $\B_g$ be the substructure of $\M$ generated by $\rho(A_{\rho\inv(g)})$ and let $\bfB:=\{\B_g:g\in G\}$. Then $\G[\bfB]=\M$ is a computably composite structure.
\end{prop}
\begin{proof}
    Since $\{S\}\cup\{A_x:x\in S\}$ are pairwise disjoint and $\rho$ is a bijection, it follows that $\{G\}\cup\{B_g:g\in G\}$ are also pairwise disjoint. It is clear that the construction of $\G[\bfB]$ yields exactly the structure $\M$. $\G$ is the substructure of $\M$ generated by the set
    $G=\{x\in M:\mu^\M(x,x)\}$ which is computable, so $\G$ is computable. It remains to verify that $\bfB$ is uniformly computable.
    
    For each $g\in G$, $\B_g$ is the substructure of $\M$ generated by the set $B_g=\{b\in M:\mu^\M(b,g)\land b\neq g\}$
    which is uniformly computable in $g$. Then $\bfB=\{\B_g\}_{g\in\G}$ is uniformly computable since a relation holds over elements of $B_g$ if and only if the relation holds over the same elements of $\M$.
\end{proof}
Next, we characterize the isomorphisms between two computably composite structures $\calS[\bfA]$ and $\G[\bfB]$. Informally, these are exactly the maps that induce an isomorphism from $\calS$ to $\G$ and induce isomorphisms between the corresponding component structures.
\begin{prop}\label{composite isos}
    Suppose $\calS[\bfA]$ and $\G[\bfB]$ are isomorphic computably composite structures. Then the isomorphisms from $\calS[\bfA]$ to $\G[\bfB]$ are exactly the maps of the form
    \[\rho=\theta\cup\bigcup_{x\in S}\psi_x\]
    where $\theta:\calS\cong\G$ and $\psi_x:\A_x\cong\B_{\theta(x)}$ for each $x\in S$.
\end{prop}
\begin{proof}
    For any such $\theta$ and $\{\psi_x\}_{x\in S}$, since $\dom(\theta)\cup\bigcup_{x\in S}\dom(\psi_x)=S\cup\bigcup_{x\in S}A_x$, $\rho$ is actually a function from the universe of $\calS[\bfA]$ to the universe of $\G[\bfB]$. Since $\theta(S)\cup\bigcup_{x\in S}\psi_x(A_x)=G\cup\bigcup_{x\in S}B_{\theta(x)}$ and $\theta$ is a bijection, $\rho$ is a bijection. It is easily checked that all relations are preserved by $\rho$, so every such $\rho$ is an isomorphism.

    Let $\rho:\calS[\bfA]\to\G[\bfB]$ be any isomorphism. We have $x\in S\Leftrightarrow\mu^{\calS[\bfA]}(x,x)$ and $g\in G\Leftrightarrow\mu^{\G[\bfB]}(g,g)$, so since $\rho$ is an isomorphism, we must have that $\theta:=\rho\restrict_S$ is an isomorphism from $\calS$ to $\G$. For any $x\in S$, $a\in A_x\Leftrightarrow\left(\mu^{\calS[\bfA]}(a,x)\land\neg\mu^{\calS[\bfA]}(a,a)\right)$ and
    \[g\in B_{\theta(x)}\Leftrightarrow\left(\mu^{\G[\bfB]}(g,\theta(x))\land\neg\mu^{\G[\bfB]}(g,g)\right)\Leftrightarrow\left(\mu^{\G[\bfB]}(g,\rho(x))\land\neg\mu^{\G[\bfB]}(g,g)\right),\]
    so since $\rho$ is an isomorphism, we must have that $\psi_x:=\rho\restrict_{A_x}$ is an isomorphism from $\A_x$ to $\B_{\theta(x)}$. As this covers the domain of $\calS[\bfA]$, we must have $\rho=\theta\cup\bigcup_{x\in S}\psi_x$.
\end{proof}

\begin{rem}\label{comp stable}
    An important case is when the base structure $\calS$ is \emph{computably stable}: for every computable copy $\G$ of $\calS$, every isomorphism from $\calS$ to $\G$ is computable. When $\calS$ is computably stable and $\calS[\bfA]$ is computably composite, Proposition~\ref{copies are composite} implies that up to computable isomorphism, the computable copies of $\calS[\bfA]$ are exactly the computably composite structures $\calS[\bfB]$ where $\A_x\cong\B_x$ for all $x\in S$. Proposition~\ref{composite isos} implies that the complexity of the isomorphisms between $\calS[\bfA]$ and $\calS[\bfB]$ depend only on the complexity of the isomorphisms between the corresponding component structures. This will be an important property of the structure defined in Section~\ref{sec:H}.
\end{rem}

\subsection{Degrees of categoricity that are never strong}\label{sec:never strong}
It is not known whether there exists a degree of categoricity that is not the strong degree of categoricity of any computable structure. We adapt a result of Bazhenov, Kalimullin, and Yamaleev which serves as motivation for the main result of this paper. For this, we define a particularly simple class of computably composite structures:
\begin{exmp}
    For each $n<\omega$, define the directed graph $\calP_n=(\{0,\dots,n-1\},E_n)$ where $E_n=\{(k,k+1):k<n-1\}$. If $\{\A_i:i<n\}$ is any collection of computable structures, it is uniformly computable since it is finite, so $\calP_n[\{i\}\times\A_i:i<n]$ is a computably composite structure (we relabel the universe $A_i$ as $\{i\}\times A_i$ to ensure disjointness). We may think of $\calP_n[\{i\}\times\A_i:i<n]$ as a directed path \emph{of the structures}:
    \[\A_0\to\A_1\to\dots\to\A_{n-1}.\]
    Since $\calP_n$ is computably stable, it follows from Proposition~\ref{copies are composite} and Proposition~\ref{composite isos} that the computable copies of $\calP_n[\{i\}\times\A_i:i<n]$ are (up to computable isomorphism) of the form $\calP_n[\{i\}\times\B_i:i<n]$ where $\B_i$ is a computable copy of $\A_i$ for each $i<n$. Each isomorphism between these paths is Turing-equivalent to a union $\bigcup_{i<n}\psi_i$ where $\psi_i:\A_i\cong\B_i$ for each $i<n$. Thus, their isomorphism spectrum is exactly $\bigcap_{i<n}\IsoSpec(\A_i,\B_i)$.
\end{exmp}
\begin{thm}[follows from Bazhenov, Kalimullin, Yamaleev
2020 \cite{BKY20}, Proposition 1b]\label{thm:BKY20}
    If $\M$ is a computable structure with spectral dimension $n<\omega$ and degree of categoricity $\bfd$, then $\calP_n[\{i\}\times\M:i<n]$ has strong degree of categoricity $\bfd$.
\end{thm}
\begin{proof}
    Since $\M$ has spectral dimension $n$, there are pairs $(\A_i,\B_i)_{i<n}$ of computable copies of $\M$ such that
    \[\bigcap_{i<n}\IsoSpec(\A_i,\B_i)=\CatSpec(\M)=\D_{\geq}(\bfd).\]
    Now $\calP_n[\A_i:i<n]$ and $\calP_n[\B_i:i<n]$ are computably composite structures as in the previous example, and are both isomorphic to $\calP_n[\{i\}\times\M:i<n]$. Thus,
    \[\IsoSpec(\calP_n[\A_i:i<n],\calP_n[\B_i:i<n])=\bigcap_{i<n}\IsoSpec(\A_i,\B_i)=\D_{\geq}(\bfd).\]
    Since every computable copy of $\calP_n[\{i\}\times\M:i<n]$ is a path $\calN_0\to\dots\to\calN_{n-1}$ with each $\calN_i$ a computable copy of $\M$, and $\M$ is $\bfd$-computably categorical, every isomorphism between these paths is computable in $\bfd$. So $\calP_n[\M]$ has degree of categoricity $\bfd$ and spectral dimension 1 witnessed by copies $\calP_n[\A_i:i<n]$ and $\calP_n[\B_i:i<n]$.
\end{proof}
\begin{defn}
    An upward-closed set of Turing degrees $\calZ\subseteq\D$ is \emph{generated} by a subset $\X\subseteq\calZ$ if $\calZ=\bigcup_{\bfd\in\X}\D_{\geq}(\bfd)$. We say that $\calZ$ is \emph{finitely generated} if it is generated by a finite set of Turing degrees. That is, if it is equal to a union of finitely-many cones.
\end{defn}
Bazhenov, Kalimullin, and Yamaleev~\cite{BKY16} show that any computable structure with a degree of categoricity and infinite spectral dimension must have a pair of copies whose isomorphism spectrum is not finitely generated.
\begin{thm}[Bazhenov, Kalimullin, Yamaleev \cite{BKY16}, Theorem 3.1]
    Let $\M$ be a computable structure with degree of categoricity $\bfd$, and infinite spectral dimension. Then there are computable copies $\A$ and $\B$ of $\M$ such that $\IsoSpec(\A,\B)$ is not finitely generated.
\end{thm}
By combining these results, we obtain the following necessary conditions for any never-strong degree of categoricity:
\begin{thm}\label{cor:never strong}
    If $\bfd$ is the degree of categoricity of a computable structure $\M$ and is not the strong degree of categoricity of any computable structure, then $\M$ has infinite spectral dimension, an infinite automorphism group, and a pair of computable copies $\A$ and $\B$ such that $\IsoSpec(\A,\B)$ is not finitely generated.
\end{thm}
Motivated by this result, we now work toward constructing pairs of computably composite structures whose isomorphism spectra are not finitely generated.

\section{An infinite base structure}\label{sec:H}
In the proof of Theorem~\ref{thm:BKY20}, we used computably composite structures on the finite base structure $\calP_n$ to construct new computable structures whose isomorphism spectra are the intersection of the original $n$ isomorphism spectra. Similar constructions yield finite unions of isomorphism spectra (for example, if $\A_1\cong\B_1$ and $\A_2\cong\B_2$, then undirected paths $\A_1-\A_2-\B_2-\B_1$ and $\A_1-\B_2-\A_2-\B_1$ have isomorphism spectrum $\IsoSpec(\A_1,\B_1)\cup\IsoSpec(\A_2,\B_2)$). In order to obtain examples of isomorphism spectra, we will show in this section that we can form arbitrary \emph{infinite} unions of isomorphism spectra. To do this, we define an infinite base structure, $\calH$, to replace the paths in the previous examples, and show that computably composite structures on $\calH$ are witnesses to the existence of the isomorphism spectra required for Theorem~\ref{thm:comp unions}.
\begin{defn}
    We define the structure $\calH=(H,\{D_i\}_{i<\omega},\{E_i\}_{i<\omega})$ as follows:

    Define $H:=[\omega]^{<\omega}\cup(\omega\times\{0,1\})$, so an element of $H$ is either a finite subset of $\omega$, or a pair $(i,a)$ where $i<\omega$ and $a<2$. By identifying finite sets with their characteristic functions, we may think of $\finsets$ as the collection of infinite binary strings with finitely-many $1$'s. We will also think of $[\omega]^{<\omega}$ as the vertices of an infinite-dimensional hypercube in which vertices $X,Y\in\finsets$ are adjacent if $|X\triangle Y|=1$.
    
    The structure $\calH$ is equipped with undirected edge relations $\{E_i\}_{i<\omega}$ forming the edges of the hypercube running along each dimension. That is, for any $X,Y\in\finsets$ and $i<\omega$,
    \[E_i(X,Y)\Leftrightarrow E_i(Y,X)\Leftrightarrow X\triangle Y=\{i\}.\]
    Figure~\ref{img:H_3_finsets} is an illustration of the 3-dimensional substructure of the hypercube using only the subsets of $\{0,1,2\}$ and the edge relations $E_0$, $E_1$, and $E_2$ between them.
    \begin{figure}[!ht]
        \centering
        \begin{tikzpicture}[align=center,xscale=0.5,yscale=0.5]
            \node (empty) at (0,0) {\footnotesize $\emptyset$};
            \node (0) at (10,0) {\footnotesize $\{0\}$};
            \node (1) at (0,10) {\footnotesize $\{1\}$};
            \node (2) at (4,4) {\footnotesize $\{2\}$};
            \node (01) at (10,10) {\footnotesize $\{0,1\}$};
            \node (02) at (14,4) {\footnotesize $\{0,2\}$};
            \node (12) at (4,14) {\footnotesize $\{1,2\}$};
            \node (012) at (14,14) {\footnotesize $\{0,1,2\}$};
            \path [-, line width=0.4mm, draw=green] (empty) edge (2);
            \path [-, line width=0.4mm, draw=red] (2) edge (02);
            \path [-, line width=0.4mm, draw=blue] (2) edge (12);
            \path [-, line width=0.4mm, draw=red] (empty) edge (0);
            \path [-, line width=0.4mm, draw=blue] (empty) edge (1);
            \path [-, line width=0.4mm, draw=blue] (0) edge (01);
            \path [-, line width=0.4mm, draw=green] (0) edge (02);
            \path [-, line width=0.4mm, draw=red] (1) edge (01);
            \path [-, line width=0.4mm, draw=green] (1) edge (12);
            \path [-, line width=0.4mm, draw=green] (01) edge (012);
            \path [-, line width=0.4mm, draw=red] (12) edge (012);
            \path [-, line width=0.4mm, draw=blue] (02) edge (012);
        \end{tikzpicture}
        \caption{Subsets of $\{0,1,2\}$ and edge relations $E_0$ (red), $E_1$ (blue), and $E_2$ (green).}
        \label{img:H_3_finsets}
    \end{figure}
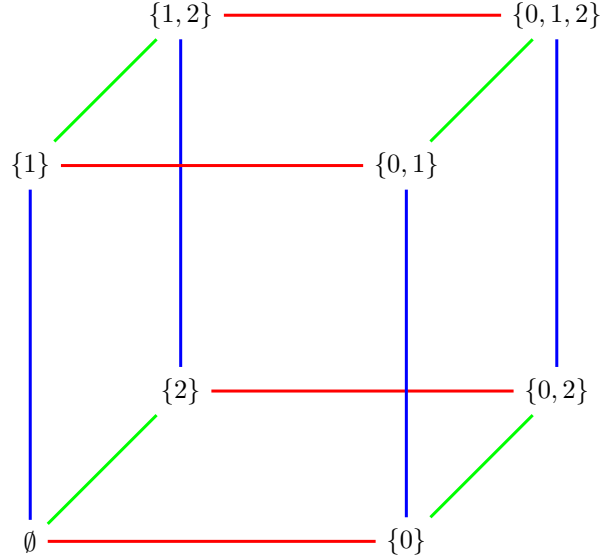
    
    Notice that for any $X\in\finsets$ and any $i<\omega$, $X\triangle\{i\}$ is the unique element of $\finsets$ with the property $E_i(X,X\triangle\{i\})$.\\

    For each $i\in\omega$, we will think of the tuples $(i,0),(i,1)\in H$ as labeling the opposing hyperfaces $\{X\in\finsets:n\notin X\}$ and $\{X\in\finsets:n\in X\}$ of the hypercube as illustrated in Figure~\ref{img:H_3_cols}. The structure $\calH$ is also equipped with $\omega$-many directed edge relations defined so that for any $X\in\finsets$, any $i<\omega$, and any $a<2$,
    \[D_i(X,(i,a))\text{ iff }X(i)=a.\]
    That is, for each finite set $X$ and each $i<\omega$, there is a directed $D_i$-edge leaving $X$. If $i\notin X$, then that edge goes to $(i,0)$, and if $i\in X$, then that edge goes to $(i,1)$. Figure~\ref{img:H_3_cols} illustrates $D_0$, $D_1$, and $D_2$.
    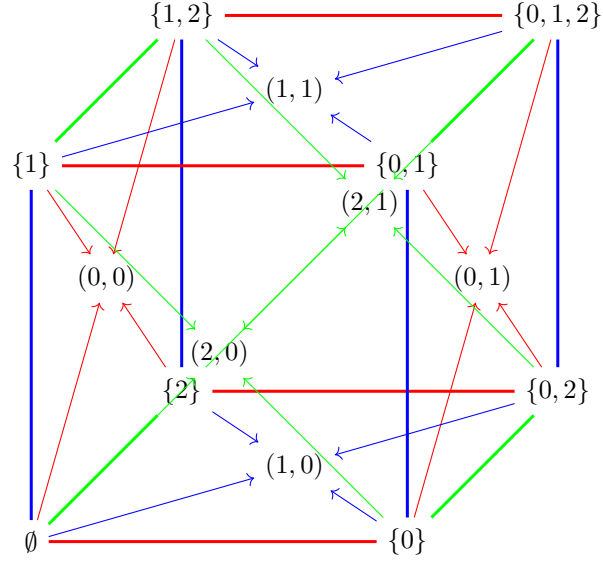
\begin{figure}[!ht]
        \centering
    \begin{tikzpicture}[align=center,xscale=0.5,yscale=0.5]
        \node (empty) at (0,0) {\footnotesize $\emptyset$};
        \node (0) at (10,0) {\footnotesize $\{0\}$};
        \node (1) at (0,10) {\footnotesize $\{1\}$};
        \node (2) at (4,4) {\footnotesize $\{2\}$};
        \node (01) at (10,10) {\footnotesize $\{0,1\}$};
        \node (02) at (14,4) {\footnotesize $\{0,2\}$};
        \node (12) at (4,14) {\footnotesize $\{1,2\}$};
        \node (012) at (14,14) {\footnotesize $\{0,1,2\}$};
        \node (0c0) at (2,7) {\footnotesize $(0,0)$};
        \node (0c1) at (12,7) {\footnotesize $(0,1)$};
        \node (1c0) at (7,2) {\footnotesize $(1,0)$};
        \node (1c1) at (7,12) {\footnotesize $(1,1)$};
        \node (2c0) at (5,5) {\footnotesize $(2,0)$};
        \node (2c1) at (9,9) {\footnotesize $(2,1)$};
        \path [-, line width=0.4mm, draw=green] (empty) edge (2);
        \path [-, line width=0.4mm, draw=red] (2) edge (02);
        \path [-, line width=0.4mm, draw=blue] (2) edge (12);
        \path [-, line width=0.4mm, draw=red] (empty) edge (0);
        \path [-, line width=0.4mm, draw=blue] (empty) edge (1);
        \path [-, line width=0.4mm, draw=blue] (0) edge (01);
        \path [-, line width=0.4mm, draw=green] (0) edge (02);
        \path [-, line width=0.4mm, draw=red] (1) edge (01);
        \path [-, line width=0.4mm, draw=green] (1) edge (12);
        \path [-, line width=0.4mm, draw=green] (01) edge (012);
        \path [-, line width=0.4mm, draw=red] (12) edge (012);
        \path [-, line width=0.4mm, draw=blue] (02) edge (012);
        \path [->, draw=red] (empty) edge (0c0);
        \path [->, draw=red] (1) edge (0c0);
        \path [->, draw=red] (2) edge (0c0);
        \path [->, draw=red] (12) edge (0c0);
        \path [->, draw=red] (0) edge (0c1);
        \path [->, draw=red] (01) edge (0c1);
        \path [->, draw=red] (02) edge (0c1);
        \path [->, draw=red] (012) edge (0c1);
        \path [->, draw=blue] (empty) edge (1c0);
        \path [->, draw=blue] (0) edge (1c0);
        \path [->, draw=blue] (2) edge (1c0);
        \path [->, draw=blue] (02) edge (1c0);
        \path [->, draw=blue] (1) edge (1c1);
        \path [->, draw=blue] (01) edge (1c1);
        \path [->, draw=blue] (12) edge (1c1);
        \path [->, draw=blue] (012) edge (1c1);
        \path [->, draw=green] (empty) edge (2c0);
        \path [->, draw=green] (0) edge (2c0);
        \path [->, draw=green] (1) edge (2c0);
        \path [->, draw=green] (01) edge (2c0);
        \path [->, draw=green] (2) edge (2c1);
        \path [->, draw=green] (02) edge (2c1);
        \path [->, draw=green] (12) edge (2c1);
        \path [->, draw=green] (012) edge (2c1);
    \end{tikzpicture}
        \caption{Subsets of $\{0,1,2\}$, edge relations $E_0$ (red), $E_1$ (blue), and $E_2$ (green), hyperfaces $\{0,1,2\}\times\{0,1\}$, and edge relations $D_0$ (red arrows), $D_1$ (blue arrows), and $D_2$ (green arrows).}
        \label{img:H_3_cols}
    \end{figure}
    
    Using a computable bijection from $\omega$ to $\finsets$, we see that $H$ is computable and that $\{E_i\}_{i<\omega}\cup\{D_i\}_{i<\omega}$ is uniformly computable, so $\calH$ is a computable structure.
\end{defn}
\begin{defn}
    For each $X\in[\omega]^{<\omega}$, we will also define a map $h_X:H\to H$ that is uniformly computable in $X$. For $Y\in[\omega]^{<\omega}$,
    \[h_X(Y):=X\triangle Y,\]
    and for $(i,a)\in\omega\times\{0,1\}$,
    \[h_X(i,a):=(i,a+X(i))=\begin{cases}
        (i,a)&\text{if }i\notin X,\\
        (i,1-a)&\text{if }i\in X.
    \end{cases}\]
\end{defn}
\begin{rem}
     Notice that $h_\emptyset$ is the identity on $H$. Fix $i<\omega$.
     \begin{enumerate}
         \item If $i\notin X$, then $h_{\{i\}}(X)=X\cup\{i\}$ and $h_{\{i\}}(X\cup\{i\})=X$.
         \item $h_{\{i\}}(i,0)=(i,1)$ and $h_{\{i\}}(i,1)=(i,0)$.
         \item If $j\neq i$, $h_{\{i\}}(j,0)=(j,0)$ and $h_{\{i\}}(j,1)=(j,1)$.
     \end{enumerate}
     This shows that $h_{\{i\}}$ is the reflection of the hypercube along the $i^\text{th}$ dimension.
\end{rem}

\subsection{Automorphisms of the base structure}
By Proposition~\ref{composite isos}, isomorphisms between copies of a computably composite structure on $\calH$ depend partly on the isomorphisms between computable copies of $\calH$. We will see that it is sufficient to consider the automorphisms of $\calH$.
\begin{lem}\label{lem:sets to sets, tuples to tuples}
    Let $\theta$ be an automorphism of $\calH$. If $X\in\finsets$, then $\theta(X)\in\finsets$. If $(i,a)\in\omega\times\{0,1\}$, then either $\theta(i,a)=(i,0)$ or $\theta(i,a)=(i,1)$.
\end{lem}
\begin{proof}
    We consider the number of inward and outward $D_i$ edges at each point in $H$:
    \begin{enumerate}
        \item If $X\in\finsets$, then $X$ has exactly one outward directed $D_i$-edge for each $i<\omega$.
        \item If $(i,a)\in\omega\times\{0,1\}$, then $(i,a)$ has no outward directed $D_j$-edges for any $j<\omega$.
        \item If $(i,a)\in\omega\times\{0,1\}$, then $(i,a)$ has infinitely-many incoming $D_i$-edges.
        \item If $(i,a)\in\omega\times\{0,1\}$, then $(i,a)$ has no incoming $D_j$-edges when $j\neq i$.
    \end{enumerate}
    As these properties must be preserved by any automorphism, the result follows.
\end{proof}
\begin{lem}\label{Lem:h_X.h_Y}
    For $X,Y\in\finsets$, $h_X\circ h_Y=h_{X\triangle Y}$. In particular, $h_X\circ h_X=h_\emptyset=id_H$ and $h_X\circ h_{\{i\}}=h_{X\cup\{i\}}$ when $i\notin X$.
\end{lem}
\begin{proof}
    For $Z\in\finsets$,
    \[h_X\circ h_Y(Z)=X\triangle(Y\triangle Z)=(X\triangle Y)\triangle Z=h_{X\triangle Y}(Z).\]
    For $(i,a)\in\omega\times\{0,1\}$,
    \[h_X\circ h_Y(i,a)=h_X(i,a+Y(i))=(i,a+Y(i)+X(i))=(i,a+(X\triangle Y)(i))=h_{X\triangle Y}(i,a).\]
\end{proof}
\begin{prop}\label{prop:h_X auto}
    For each $X\in\finsets$, the map $h_X$ is an automorphism of $\calH$.
\end{prop}
\begin{proof}
    By Lemma~\ref{Lem:h_X.h_Y}, $h_X\circ h_X=id_H$, so $h_X$ is bijective. If $i\notin X$, then $i\in Y\Leftrightarrow i\in(X\triangle Y)$. If $i\in X$, then $i\in Y\Leftrightarrow i\notin(X\triangle Y)$. So in general, we have
    \[Y(i)=a\Leftrightarrow (X\triangle Y)(i)=a+X(i).\]
    The following equivalences show that $h_X$ preserves all relations $E_i$ and $D_i$ in $\calH$ and so is an automorphism:
    \begin{align*}
        D_i(Y,(i,a))&\Leftrightarrow Y(i)=a\\
        &\Leftrightarrow (X\triangle Y)(i)=a+X(i)\\
        &\Leftrightarrow h_X(Y)(i)=a+X(i)\\
        &\Leftrightarrow D_i(h_X(Y),(i,a+X(i)))\\
        &\Leftrightarrow D_i(h_X(Y),h_X(i,a)).
    \end{align*}
    \begin{align*}
        E_i(Y,Z)&\Leftrightarrow Y\triangle Z=\{i\}\\
        &\Leftrightarrow(Y\triangle X)\triangle(Z\triangle X)=\{i\}\\
        &\Leftrightarrow E_i(Y\triangle X,Z\triangle X)\\
        &\Leftrightarrow E_i(h_X(Y),h_X(Z)).
    \end{align*}
\end{proof}
\begin{prop}\label{prop:h_X unique auto}
    For each $X\in\finsets$, the map $h_X$ is the unique automorphism sending $\emptyset$ to $X$.
\end{prop}
\begin{proof}
    Every $h_X$ is an automorphism by Proposition~\ref{prop:h_X auto}. Suppose $\theta:\calH\cong\calH$ is an automorphism and that $\theta(\emptyset)=X$. By Lemma~\ref{lem:sets to sets, tuples to tuples}, for each $i<\omega$ there is an $a_i\in\{0,1\}$ such that $\theta(i,0)=(i,a_i)$ and $\theta(i,1)=(i,1+a_i)$. Since $(\emptyset,(i,0))\in D_i$ and $\theta$ is an automorphism, we have $(\theta(\emptyset),\theta(i,0))=(X,(i,a_i))\in D_i$ for each $i<\omega$. But then by the definition of $D_i$, $a_i=X(i)$ for all $i$, so
    \[\theta(i,0)=(i,a_i)=(i,X(i))=h_X(i,0)\]
    and
    \[\theta(i,1)=(i,1+a_i)=(i,1+X(i))=h_X(i,1).\]
    Thus, $\theta(i,a)=h_X(i,a)$ for all $(i,a)\in\omega\times\{0,1\}$.\\

    Let $Y\in\finsets$, so by Lemma~\ref{lem:sets to sets, tuples to tuples}, $Z:=\theta(Y)\in\finsets$. Since $\theta(i,a)=h_X(i,a)=(i,a+X(i))$, we have
    \[Y(i)=a\Leftrightarrow D_i(Y,(i,a))\Leftrightarrow D_i(Z,(i,a+X(i)))\Leftrightarrow Z(i)=a+X(i),\]
    so $Z(i)=Y(i)+X(i)$ for all $i<\omega$, and so $\theta(Y)=Z=X\triangle Y=h_X(Y)$. Since $Y\in\finsets$ is arbitrary, we have $\theta=h_X$.
\end{proof}
\begin{cor}\label{H autos}
    The set of automorphisms of $\calH$ is $\left\{h_X:X\in\finsets\right\}$.
\end{cor}
\begin{proof}
    By Lemma~\ref{lem:sets to sets, tuples to tuples}, every automorphism of $\calH$ sends $\emptyset$ to some $X\in[\omega]^{<\omega}$, so by Proposition~\ref{prop:h_X unique auto} is equal to $h_X$. The result then follows from Proposition~\ref{prop:h_X auto}.
\end{proof}
\begin{rem}
    By Lemma~\ref{Lem:h_X.h_Y}, we have that if $X=\{x_1,x_2,\dots,x_k\}\in\finsets$, then
    \[h_X=h_{\{x_1\}}\circ h_{\{x_2\}}\circ\dots\circ h_{\{x_k\}}.\]
    Since $h_{\{i\}}$ is the reflection of the hypercube along the $i^\text{th}$ dimension, each $h_X$ is a finite composition of reflections along the dimensions corresponding to the elements of $X$.
\end{rem}

From Remark~\ref{comp stable}, the following proposition means that if $\calH[\bfA]$ is a computably composite structure on $\calH$, the computable copies of $\calH[\bfA]$ are up to computable isomorphism exactly the computably composite structures of the form $\calH[\bfB]$ where $\B_z\cong\A_z$ for all $z\in H$. Hence, the complexity of an isomorphism between two computably composite structures on $\calH$ depends only on the isomorphisms between the corresponding component structures.

\begin{prop}\label{H comp isos}
    $\calH$ is computably stable. That is, if $\G$ is a computable copy of $\calH$, then every isomorphism from $\calH$ to $\G$ is computable.
\end{prop}
\begin{proof}
    Let $f$ be any isomorphism from $\calH$ to $\G$. We may assume that $\G$ has universe $\omega$. Define $n_\emptyset:=f(\emptyset)$. By Proposition~\ref{prop:h_X unique auto}, the automorphisms of $\calH$ are determined by where $\emptyset$ is mapped to. Thus, $f$ is determined by $f(\emptyset)$. We now computably build $f$ out from $\emptyset$.

    Suppose we know $f(X)\in\omega$ for some $X\in\finsets$ and suppose $n\notin X$. Then $X\cup\{n\}$ is the unique element of $\finsets$ with the property $E^\calH_n(X,X\cup\{n\})$. Thus, $f(X\cup\{n\})$ is the unique element of $\omega$ with the property $E^\G_n(f(X),f(X\cup\{n\}))$ and can be found by iterating through $\omega$ in finitely-many steps since $E^\G_n$ is computable. Since we know $f(\emptyset)$, we can find $f(Z)$ for any $Z\in\finsets$ by iterating this process once for each of the (finitely-many) elements of $Z$.
    
    Fix $i\in\omega$. Then $(i,0)$ is the unique element of $H$ with the property $D_i^\calH(\emptyset,(i,0))$. Thus, $f(i,0)$ is the unique element of $\omega$ with the property $D_i^\G(f(\emptyset),f(i,0))$. Since we know $f(\emptyset)$, $f(i,0)$ can be found by iterating through $\omega$ in finitely-many steps. Similarly, $f(i,1)$ is the unique element of $\omega$ with the property $D_i^\G(\{i\},(i,1))$. Since we can computably find $f(\{i\})$, we can find $f(i,1)$ by iterating through $\omega$ in finitely-many steps.
\end{proof}
\begin{cor}\label{H comp cat}
    $\calH$ is computably categorical.
\end{cor}

\subsection{Computable unions of isomorphism spectra}
We now show that the class of isomorphism spectra is closed under computable unions:
\begin{thm}\label{thm:comp unions}
    Given any two uniformly computable collections of copies $\bfA=\{\A_i:i<\omega\}$ and $\bfB=\{\B_i:i<\omega\}$ such that for each $i$, $\A_i\cong\B_i$, there exists a structure with two computable copies $\M\cong\calN$ where $\IsoSpec(\M,\calN)=\bigcup_{i<\omega}\IsoSpec(\A_i,\B_i)$.
\end{thm}

Fix two arbitrary uniformly computable collections of copies $\bfA=\{\A_i:i<\omega\}$ and $\bfB=\{\B_i:i<\omega\}$ with arbitrary uniformly computable relational languages $\{\Lang_i:i<\omega\}$ such that for each $i$, $\A_i\cong\B_i$. We will define $\M$ and $\calN$, two computably composite structures on $\calH$, and then verify that they satisfy the statement of Theorem~\ref{thm:comp unions}.
\begin{defn}\label{M N defn}
    We first define uniformly computable collections of copies $\bfM=\{\M_z:z\in H\}$ and $\bfN=\{\calN_z:z\in H\}$. Each $\M_z$ and each $\calN_z$ will be of the form $\{z\}\times\A_k$ or $\{z\}\times\B_k$ for some $k$, ensuring that the universes are disjoint (see Figure~\ref{img:H_3_AB}):
    \begin{figure}[!ht]
        \centering
        \begin{tikzpicture}[align=center,xscale=0.43,yscale=0.43]
            \node (empty) at (0,0) {\tiny $\{\emptyset\}\times\mathcal{A}_0$};
            \node (0) at (10,0) {\tiny $\{\{0\}\}\times\mathcal{B}_0$};
            \node (1) at (0,10) {\tiny $\{\{1\}\}\times\mathcal{B}_0$};
            \node (2) at (4,4) {\tiny $\{\{2\}\}\times\mathcal{B}_0$};
            \node (01) at (10,10) {\tiny $\{\{0,1\}\}\times\mathcal{A}_0$};
            \node (02) at (14,4) {\tiny $\{\{0,2\}\}\times\mathcal{A}_0$};
            \node (12) at (4,14) {\tiny $\{\{1,2\}\}\times\mathcal{A}_0$};
            \node (012) at (14,14) {\tiny $\{\{0,1,2\}\}\times\mathcal{B}_0$};
            \node (0c0) at (2,7) {\tiny $\{(0,0)\}\times\mathcal{A}_1$};
            \node (0c1) at (12,7) {\tiny $\{(0,1)\}\times\mathcal{B}_1$};
            \node (1c0) at (7,2) {\tiny $\{(1,0)\}\times\mathcal{A}_2$};
            \node (1c1) at (7,12) {\tiny $\{(1,1)\}\times\mathcal{B}_2$};
            \node (2c0) at (5,5) {\tiny $\{(2,0)\}\times\mathcal{A}_3$};
            \node (2c1) at (9,9) {\tiny $\{(2,1)\}\times\mathcal{B}_3$};
            \path [-, line width=0.3mm, draw=green] (empty) edge (2);
            \path [-, line width=0.3mm, draw=red] (2) edge (02);
            \path [-, line width=0.3mm, draw=blue] (2) edge (12);
            \path [-, line width=0.3mm, draw=red] (empty) edge (0);
            \path [-, line width=0.3mm, draw=blue] (empty) edge (1);
            \path [-, line width=0.3mm, draw=blue] (0) edge (01);
            \path [-, line width=0.3mm, draw=green] (0) edge (02);
            \path [-, line width=0.3mm, draw=red] (1) edge (01);
            \path [-, line width=0.3mm, draw=green] (1) edge (12);
            \path [-, line width=0.3mm, draw=green] (01) edge (012);
            \path [-, line width=0.3mm, draw=red] (12) edge (012);
            \path [-, line width=0.3mm, draw=blue] (02) edge (012);
            \path [->, draw=red] (empty) edge (0c0);
            \path [->, draw=red] (1) edge (0c0);
            \path [->, draw=red] (2) edge (0c0);
            \path [->, draw=red] (12) edge (0c0);
            \path [->, draw=red] (0) edge (0c1);
            \path [->, draw=red] (01) edge (0c1);
            \path [->, draw=red] (02) edge (0c1);
            \path [->, draw=red] (012) edge (0c1);
            \path [->, draw=blue] (empty) edge (1c0);
            \path [->, draw=blue] (0) edge (1c0);
            \path [->, draw=blue] (2) edge (1c0);
            \path [->, draw=blue] (02) edge (1c0);
            \path [->, draw=blue] (1) edge (1c1);
            \path [->, draw=blue] (01) edge (1c1);
            \path [->, draw=blue] (12) edge (1c1);
            \path [->, draw=blue] (012) edge (1c1);
            \path [->, draw=green] (empty) edge (2c0);
            \path [->, draw=green] (0) edge (2c0);
            \path [->, draw=green] (1) edge (2c0);
            \path [->, draw=green] (01) edge (2c0);
            \path [->, draw=green] (2) edge (2c1);
            \path [->, draw=green] (02) edge (2c1);
            \path [->, draw=green] (12) edge (2c1);
            \path [->, draw=green] (012) edge (2c1);
        \end{tikzpicture}
        \hspace{1em}
    \begin{tikzpicture}[align=center,xscale=0.43,yscale=0.43]
        \node (empty) at (0,0) {\tiny $\{\emptyset\}\times\mathcal{B}_0$};
        \node (0) at (10,0) {\tiny $\{\{0\}\}\times\mathcal{A}_0$};
        \node (1) at (0,10) {\tiny $\{\{1\}\}\times\mathcal{A}_0$};
        \node (2) at (4,4) {\tiny $\{\{2\}\}\times\mathcal{A}_0$};
        \node (01) at (10,10) {\tiny $\{\{0,1\}\}\times\mathcal{B}_0$};
        \node (02) at (14,4) {\tiny $\{\{0,2\}\}\times\mathcal{B}_0$};
        \node (12) at (4,14) {\tiny $\{\{1,2\}\}\times\mathcal{B}_0$};
        \node (012) at (14,14) {\tiny $\{\{0,1,2\}\}\times\mathcal{A}_0$};
        \node (0c0) at (2,7) {\tiny $\{(0,0)\}\times\mathcal{A}_1$};
        \node (0c1) at (12,7) {\tiny $\{(0,1)\}\times\mathcal{B}_1$};
        \node (1c0) at (7,2) {\tiny $\{(1,0)\}\times\mathcal{A}_2$};
        \node (1c1) at (7,12) {\tiny $\{(1,1)\}\times\mathcal{B}_2$};
        \node (2c0) at (5,5) {\tiny $\{(2,0)\}\times\mathcal{A}_3$};
        \node (2c1) at (9,9) {\tiny $\{(2,1)\}\times\mathcal{B}_3$};
        \path [-, line width=0.3mm, draw=green] (empty) edge (2);
        \path [-, line width=0.3mm, draw=red] (2) edge (02);
        \path [-, line width=0.3mm, draw=blue] (2) edge (12);
        \path [-, line width=0.3mm, draw=red] (empty) edge (0);
        \path [-, line width=0.3mm, draw=blue] (empty) edge (1);
        \path [-, line width=0.3mm, draw=blue] (0) edge (01);
        \path [-, line width=0.3mm, draw=green] (0) edge (02);
        \path [-, line width=0.3mm, draw=red] (1) edge (01);
        \path [-, line width=0.3mm, draw=green] (1) edge (12);
        \path [-, line width=0.3mm, draw=green] (01) edge (012);
        \path [-, line width=0.3mm, draw=red] (12) edge (012);
        \path [-, line width=0.3mm, draw=blue] (02) edge (012);
        \path [->, draw=red] (empty) edge (0c0);
        \path [->, draw=red] (1) edge (0c0);
        \path [->, draw=red] (2) edge (0c0);
        \path [->, draw=red] (12) edge (0c0);
        \path [->, draw=red] (0) edge (0c1);
        \path [->, draw=red] (01) edge (0c1);
        \path [->, draw=red] (02) edge (0c1);
        \path [->, draw=red] (012) edge (0c1);
        \path [->, draw=blue] (empty) edge (1c0);
        \path [->, draw=blue] (0) edge (1c0);
        \path [->, draw=blue] (2) edge (1c0);
        \path [->, draw=blue] (02) edge (1c0);
        \path [->, draw=blue] (1) edge (1c1);
        \path [->, draw=blue] (01) edge (1c1);
        \path [->, draw=blue] (12) edge (1c1);
        \path [->, draw=blue] (012) edge (1c1);
        \path [->, draw=green] (empty) edge (2c0);
        \path [->, draw=green] (0) edge (2c0);
        \path [->, draw=green] (1) edge (2c0);
        \path [->, draw=green] (01) edge (2c0);
        \path [->, draw=green] (2) edge (2c1);
        \path [->, draw=green] (02) edge (2c1);
        \path [->, draw=green] (12) edge (2c1);
        \path [->, draw=green] (012) edge (2c1);
    \end{tikzpicture}
        \caption{The first three dimensions of $\M=\calH[\bfM]$ (left) and $\calN=\calH[\bfN]$ (right).}
        \label{img:H_3_AB}
    \end{figure}
    \begin{enumerate}
        \item For $i\in\omega$, define $\M_{(i,0)}=\calN_{(i,0)}:=\{(i,0)\}\times\A_{i+1}$.
        \item For $i\in\omega$, define $\M_{(i,1)}=\calN_{(i,1)}:=\{(i,1)\}\times\B_{i+1}$.
        \item For $X\in\finsets$, define $\M_X:=\begin{cases}
            \{X\}\times\A_0&\text{if $|X|$ is even},\\
            \{X\}\times\B_0&\text{if $|X|$ is odd}.
        \end{cases}$
        \item For $X\in\finsets$, define $\calN_X:=\begin{cases}
            \{X\}\times\B_0&\text{if $|X|$ is even},\\
            \{X\}\times\A_0&\text{if $|X|$ is odd}.
        \end{cases}$
    \end{enumerate}
    For $X,Y\in\finsets$, define functions $\alpha_{X,Y}$ and $\beta_{X,Y}$ so that for $a\in A_0$, $\alpha_{X,Y}(X,a):=(Y,a)$, and for $b\in B_0$, $\beta_{X,Y}(X,b):=(Y,b)$. Notice that $\alpha_{X,Y}$ is a computable isomorphism from $\{X\}\times\A_0$ to $\{Y\}\times\A_0$ and $\beta_{X,Y}$ is a computable isomorphism from $\{X\}\times\B_0$ to $\{Y\}\times\B_0$. Thus, every copy in $\{\M_X,\calN_X:X\in\finsets\}$ is computably isomorphic to either $\M_\emptyset$ via $\alpha_{X,\emptyset}$ or to $\calN_\emptyset$ via $\beta_{X,\emptyset}$. The copies $\M_\emptyset$ and $\calN_\emptyset$ are isomorphic, but possibly not computably so. Moreover, $\alpha_{X,Y}$ and $\beta_{X,Y}$ are uniformly computable in $X$ and $Y$. Finally, define the computably composite structures
    \[\M:=\calH[\bfM]\text{ and }\calN:=\calH[\bfN].\]
    For every $z\in H$, $\M_z$ and $\calN_z$ are isomorphic, so the computably composite structures $\M$ and $\calN$ are isomorphic.
\end{defn}
\begin{prop}\label{H isos}
    The isomorphisms from $\M$ to $\calN$ are exactly the maps
    \[h_X\cup\bigcup_{z\in H}\psi_z\]
    where $X\in\finsets$ and $\psi_z:\M_z\cong\calN_{h_X(z)}$ for each $z\in H$.
\end{prop}
\begin{proof}
    This follows directly from Proposition~\ref{composite isos} and Proposition~\ref{H autos}.
\end{proof}
The rest of this section discusses the Turing degrees of these isomorphisms and the isomorhpism spectrum $\IsoSpec(\M,\calN)$.
\begin{lem}\label{Mempty iso comp M iso}
    Let $\theta:\A_0\cong\B_0$. Then there is an isomorphism $\rho:\M\cong\calN$ such that $\rho\equiv_T\theta$.
\end{lem}
\begin{proof}
    Since $\M_\emptyset=\{\emptyset\}\times\A_0$ and $\calN_\emptyset=\{\emptyset\}\times\B_0$, the function $\wh{\theta}(\emptyset,a):=(\emptyset,\theta(a))$ is an isomorphism from $\M_\emptyset$ to $\calN_{\emptyset}$ and $\wh{\theta}\equiv_T\theta$. Define the isomorphism $\rho=h_X\cup\bigcup_{z\in H}\psi_z$ as follows:
    
    Take $X=\emptyset$, so $h_X=h_\emptyset=id_H$. We need to define each $\psi_z:\M_z\cong\calN_{h_X(z)}=\calN_z$:
    \begin{enumerate}
        \item For $z=(i,a)\in\omega\times\{0,1\}$, $\M_{(i,a)}=\calN_{(i,a)}$, so we may take $\psi_{(i,a)}:=id_{M_{(i,a)}}$.
        \item For $z=Y\in\finsets$ with $|Y|$ even we have $\alpha_{Y,\emptyset}:\M_Y\cong\M_\emptyset$ and $\beta_{\emptyset,Y}:\calN_\emptyset\cong\calN_Y$, so let
        \[\psi_Y:=\beta_{\emptyset,Y}\circ\wh{\theta}\circ\alpha_{Y,\emptyset}:\M_Y\cong\calN_Y=\calN_{h_X(Y)}.\]
        \item For $z=Y\in\finsets$ with $|Y|$ odd we have $\beta_{Y,\emptyset}:\M_Y\cong\calN_\emptyset$ and $\alpha_{\emptyset,Y}:\M_\emptyset\cong\calN_Y$, so let
        \[\psi_Y:=\alpha_{\emptyset,Y}\circ\wh{\theta}\inv\circ\beta_{Y,\emptyset}:\M_Y\cong\calN_Y=\calN_{h_X(Y)}.\]
    \end{enumerate}
    Since $\alpha_{X,Y}$ and $\beta_{X,Y}$ (as in Definition~\ref{M N defn}) are uniformly computable in $X$ and $Y$, the collection $\{\psi_Y:Y\in\finsets\}$ is uniformly computable in $\wh{\theta}$ and $\psi_Y\equiv_T\wh{\theta}$ for each $Y$. Thus, $\wh{\theta}\equiv_T\bigcup_{z\in H}\psi_z$. Since $h_\emptyset=id_H$ is computable, we have $\rho=id_\emptyset\cup\bigcup_{z\in H}\psi_z\equiv_T\wh{\theta}\equiv_T\theta$.
\end{proof}
\begin{lem}\label{Mi0 iso comp M iso}
    Fix $i<\omega$ and let $\theta:\A_{i+1}\cong\B_{i+1}$. Then there is an isomorphism $\rho:\M\cong\calN$ such that $\rho\equiv_T\theta$.
\end{lem}
\begin{proof}
    Since $\M_{i,0}=\{(i,0)\}\times\A_{i+1}$ and $\M_{i,1}=\{(i,1)\}\times\B_{i+1}$, the function $\wh{\theta}((i,0),a):=((i,1),\theta(a))$ is an isomorphism from $\M_{(i,0)}$ to $\M_{(i,1)}$ and $\wh{\theta}\equiv_T\theta$. Define the isomorphism $\rho=h_X\cup\bigcup_{z\in H}\psi_z$ as follows:
    
    Take $X=\{i\}$, so $h_X:=h_{\{i\}}$ which permutes $(i,0)$ with $(i,1)$, permutes each $Y\in\finsets$ with $Y\triangle\{i\}$, and fixes all $(j,a)$ where $j\neq i$. We need to define each $\psi_z:\M_z\cong\calN_{h_{\{i\}}(z)}$.
    \begin{enumerate}
        \item We have $\calN_{h_{\{i\}}(i,0)}=\calN_{(i,1)}=\M_{(i,1)}$ and $\calN_{h_{\{i\}}(i,1)}=\calN_{(i,0)}=\M_{(i,0)}$, so we may take $\psi_{(i,0)}:=\wh{\theta}$ and $\psi_{(i,1)}:=\wh{\theta}\inv$.
        \item For $j\neq i$ and $a\in\{0,1\}$,  $\calN_{h_{\{i\}}(j,a)}=\calN_{(j,a)}=\M_{(j,a)}$, so let $\psi_{(j,a)}:=id_{\M_{(j,a)}}$.
        \item For $|Y|$ even, take $\psi_Y:=\alpha_{Y,Y\triangle\{i\}}:\M_Y\cong\calN_{Y\triangle\{i\}}=\calN_{h_{\{i\}}(Y)}$.
        \item For $|Y|$ odd, take $\psi_Y:=\beta_{Y,Y\triangle\{i\}}:\M_Y\cong\calN_{Y\triangle\{i\}}=\calN_{h_{\{i\}}(Y)}$.
    \end{enumerate}
    Then all $\psi_z$ are uniformly computable in $z$ except for $\psi_{(i,0)}=\wh{\theta}$ and $\psi_{(i,1)}=\wh{\theta}\inv$. Since $h_i$ is computable, $\rho=h_{\{i\}}\cup\bigcup_{z\in H}\psi_z\equiv_T\wh{\theta}\equiv_T\theta$.
\end{proof}
\begin{lem}\label{M iso comp component iso}
    Let $\rho:\M\cong\calN$ be any isomorphism. Then there is an $n<\omega$ and $\theta:\A_n\cong\B_n$ such that $\rho\geq_T\theta$.
\end{lem}
\begin{proof}
    By Proposition~\ref{H isos}, $\rho=h_X\cup\bigcup_{z\in H}\psi_z$ for some $X\in\finsets$ and $\psi_z:\M_z\cong\calN_{h_X(z)}$ for each $z\in H$. If $X=\emptyset$, then $h_X=id_H$, so $\rho\geq_T\rho\restrict_{M_\emptyset}=\psi_\emptyset:\M_\emptyset\cong\calN_\emptyset$. Since $\M_\emptyset$ and $\calN_\emptyset$ are computably isomorphic to $\A_0$ and $\B_0$ respectively, $\rho$ computes an $\A_0\cong\B_0$ isomorphism.
    
    Now suppose $X\neq\emptyset$ and pick any $i\in X$. Then $h_X$ permutes $(i,0)$ and $(i,1)$, so
    \[\rho\geq_T\rho\restrict_{M_{(i,0)}}=\psi_{(i,0)}:\M_{(i,0)}\cong\calN_{h_X(i,0)}=\calN_{(i,1)}.\]
    Since $\M_{(i,0)}$ and $\calN_{(i,1)}$ are computably isomorphic to $\A_{i+1}$ and $\B_{i+1}$ respectively, $\rho$ computes an $\A_{i+1}\cong\B_{i+1}$ isomorphism.
\end{proof}
We can now describe the isomorphism spectrum of $\M$ and $\calN$ in terms of the isomorphism spectra of their components:
\begin{prop}\label{Isospec(M,N) is union}
    $\IsoSpec(\M,\calN)=\bigcup_{n<\omega}\IsoSpec(\A_i,\B_i)$.
\end{prop}
\begin{proof}
    It follows from Lemma~\ref{Mempty iso comp M iso} and Lemma~\ref{Mi0 iso comp M iso} that $\bigcup_{i<\omega}\IsoSpec(\A_i,\B_i)\subseteq\IsoSpec(\M,\calN)$. It follows from Lemma~\ref{M iso comp component iso} that $\IsoSpec(\M,\calN)\subseteq\bigcup_{n<\omega}\IsoSpec(\A_i,\B_i)$.
\end{proof}
As $\{\A_i:i<\omega\}$ and $\{\B_i:i<\omega\}$ were taken to be arbitrary, Theorem~\ref{thm:comp unions} follows.

\subsection{An isomorphism spectrum that is not finitely generated}\label{sec:thomason}
Recall that we say that an isomorphism spectrum is \emph{finitely generated} if it is equal to the upward-closure of a finite set of Turing degrees. We apply Theorem~\ref{thm:comp unions} to results of Thomason~\cite{Tho71} to construct an isomorphism spectrum that is not finitely generated.

We will consider the computable copies $(\omega,<_X)$ of $(\omega,<)$ from Example~\ref{exmp:lt} that encode \ce sets $X\subseteq\omega$. Recall that for each \ce set $X$, there is a unique isomorphism $f:(\omega,<)\cong(\omega,<_X)$ and $f\equiv_T X$.

We use a result of Thomason~\cite{Tho71} to construct a uniformly computable family of copies of $(\omega,<)$ so that the collection of isomorphisms between these structures and the usual copy, $(\omega,<)$, have pairwise disjoint Turing degrees.
\begin{thm}[Thomason, 1971 \cite{Tho71}]
    There is a sublattice of the c.e. degrees isomorphic to the lattice of all finite sets of natural numbers.
\end{thm}
To prove this, Thomason constructs a uniformly \ce collection of sets $\{X_n:n<\omega\}$, each with a fixed computable enumeration. For each $n<\omega$ and finite $F\subseteq\omega$, they define
\[X_F:=\bigoplus_{m\in F}X_m\text{, }\bfd_n:=\deg_T(X_n)\text{, and }\bfd_F:=\deg_T(X_F)=\bigoplus_{m\in F}\bfd_m\]
and show that for all $F,G\in\finsets$,
\begin{enumerate}
    \item $\bfd_F\leq\bfd_G\Leftrightarrow F\subseteq G$,
    \item $\bfd_F\oplus\bfd_G=\bfd_{F\cup G}$,
    \item $\bfd_F\wedge\bfd_G$ exists and $\bfd_F\wedge\bfd_G=\bfd_{F\cap G}$.
\end{enumerate}
\begin{rem}
    The collection $\{(\omega,<_{X_n}):n<\omega\}$ is uniformly computable, and the set of degrees $\{\bfd_n:n<\omega\}$ is pairwise incomparable. Moreover $\bfd_n=\deg_T(X_n)$ is the Turing degree of the unique isomorphism from $(\omega,<)$ to $(\omega,<_{X_n})$.
\end{rem}
\begin{thm}\label{non-fg isospec}
    Define the uniformly computable collections $\bfA:=\{\A_n:n<\omega\}$ and $\bfB:=\{\B_n:n<\omega\}$ where for all $n<\omega$, $\A_n=(\omega,<)$, and $\B_n=(\omega,<_{X_n})$. Let $\M$ and $\calN$ be the two computable copies from Theorem~\ref{thm:comp unions}. Then $\IsoSpec(\M,\calN)$ is not finitely generated.
\end{thm}
\begin{proof}
    By Theorem~\ref{thm:comp unions},
    \[\IsoSpec(\M,\calN)=\bigcup_{n<\omega}\IsoSpec\big((\omega,<),(\omega,<_{X_n})\big)=\bigcup_{n<\omega}\D_{\geq}(\bfd_n)\]
    which is not finitely generated since the degrees $\{\bfd_n:n<\omega\}$ are pairwise incomparable.
\end{proof}

\section{Automorphism spectra of computably composite structures}\label{sec:autospec}
Harizanov, Morozov, and Miller \cite{HMM10} have considered similar questions about \emph{automorphism spectra} which have some interesting connections to our results. We would like to thank an anonymous referee for bringing this work to our attention.
\begin{defn}[Harizanov, Morozov, Miller \cite{HMM10} Definition~1.1]
    Let $\M$ be a computable structure. The \emph{automorphism spectrum} of $\M$ is the set
    \[\text{AutSp}^*(\M)=\{\deg_T(f):f\in\Aut(\M)\minus\{1_\M\}\}\]
    where $1_\M$ is the identity automorphism on $\M$.
\end{defn}
They also consider the upward-closures of automorphism spectra under $\leq_T$ for which we will write
\[\text{AutSp}^*_\geq(\M)=\{\bfd:\exists f\in\Aut(\M)\minus\{1_\M\}[f\leq_T\bfd]\}.\]
Thus, AutSp$^*_\geq(\M)$ is equal to $\IsoSpec(\M,\M)$ with the omission of the degrees that only compute the trivial isomorphism $1_\M:\M\cong\M$. It is known that the upward-closure of an automorphism spectrum is itself an automorphism spectrum:
\begin{prop}[Harizanov, Morozov, Miller \cite{HMM10} Proposition~7.1]
    Let $\A$ be any computable nonrigid structure. Then there exists a computable structure $\calC$ such that
    \[\text{AutSp}^*(\calC)=\text{AutSp}^*_\geq(\A).\]
\end{prop}
No automorphism spectrum has the form $\{\bfd_1,\bfd_2\}$ where $\bfd_1$ and $\bfd_2$ are incomparable (\cite{HMM10} Proposition~6.1). Since there are incomparable degrees $\bfd_1$ and $\bfd_2$ such that $\{\bfd_1\}$ and $\{\bfd_2\}$ are automorphism spectra (see \cite{HMM10} Proposition~2.1), the class of automorphism spectra is not closed under even finite unions. However, the following observation uses a uniform infinitary version of the proof of Lemma~4.1 in \cite{HMM10} to show that the class of \emph{upward-closures} of automorphism spectra is closed under computable union, satisfying the analog of our Theorem~\ref{thm:comp unions}:
\begin{obs}
    If $\bfA=\{\A_i:i<\omega\}$ is a uniformly computable collection of copies and $\calO$ is the structure with domain $\omega$ and unary relations $U_n^\calO=\{n\}$ for each $n<\omega$, then $\text{AutSp}^*_\geq(\calO[\bfA])=\bigcup_{i<\omega}\text{AutSp}^*_\geq(\A_i)$.
\end{obs}

\section{Categoricity spectra of some computably composite structures}\label{sec:catspec}
We conclude by considering some basic properties of the categoricity spectra of computably composite structures on $\calH$. We will consider structures of the form $\calH[\{z\}\times\A:z\in H]$, which we will denote by $\calH[\A]$, and their computable copies up to computable isomorphism. That is, we are concerned with the case where all components are isomorphic to a particular computable structure $\A$. For example, the structures in Theorem~\ref{non-fg isospec} are computable copies of ${\calH[(\omega,<)]}$.

We may assume that computable copies of $\calH[\A]$ have the form $\calH[\bfB]$ where $\bfB=\{\B_z:z\in H\}$ is a uniformly computable collection of copies of $\A$. The isomorphisms from $\calH[\A]$ to $\calH[\bfB]$ are the functions $\rho=h_X\cup\bigcup_{z\in H}\psi_z$ where $X\in\finsets$ and $\psi_z:\{z\}\times\A\cong\B_{h_X(z)}$ for all $z\in H$. $\calH[\A]$ is $\bfd$-computably categorical if and only if $\bfd$ computes some such $\rho$ for every uniformly computable collection of copies $\{\B_z:z\in H\}$. That is, $\A$ must be $\bfd$-computably categorical in a uniform way.
\begin{defn}
    We will say that $\A$ is \emph{uniformly $\bfd$-computably categorical} if there is a partial $\bfd$-computable binary function $g$ such that whenever $\varphi_e$ is the atomic diagram of a copy $\wh{\A}$ of $\A$, $g(e,\cdot)$ is defined and is an isomorphism from $\A$ to $\wh{\A}$. We define the \emph{uniform categoricity spectrum} of $\A$ and \emph{degree of uniform categoricity} of $\A$ in the analogous ways.
\end{defn}
It is not hard to see that if $\A$ is uniformly $\bfd$-computably categorical, then $\calH[\A]$ is $\bfd$-computably categorical. In some examples, the opposite implication holds as well:
\begin{exmp}
    The simplest case is when $\A$ is uniformly computably categorical, so $\calH[\A]$ is computably categorical. Then the categoricity spectrum of $\calH[\A]$ and the uniform categoricity spectrum of $\A$ are both $\D$, the set of all Turing degrees.
\end{exmp}
\begin{exmp}
    The structure $(\omega,<)$ has strong degree of categoricity $\bfzero'$ and also uniform degree of categoricity $\bfzero'$. Every isomorphism between the copies $\calH[(\omega,<)]$ and $\calH[(\omega,<_{\emptyset'})]$ as defined in Example~\ref{exmp:lt} computes $\bfzero'$, so $\calH[(\omega,<)]$ has strong degree of categoricity $\bfzero'$, as do the structures in Theorem~\ref{non-fg isospec}.

    By the same argument, if $\A$ has strong degree of categoricity $\bfd$ and uniform degree of categoricity $\bfd$, then $\calH[\A]$ has strong degree of categoricity $\bfd$.
\end{exmp}
\begin{exmp}
    The structure $(\omega,S)$ where $S(n)=n+1$ is computably categorical, but has uniform degree of categoricity $\bfzero'$. We can construct a uniformly computable collection of copies $\{\A_n:n<\omega\}$ each with universe $\omega$, such that the $S$-least element of $\A_n$ is $0$ if and only if $n\notin\emptyset'$. Then defining $\alpha:H\to\omega$ so that $\alpha(X)=0$ for $X\in\finsets$ and $\alpha(i,a)=i+1$ for $i<\omega$ and $a<2$, the collection $\{\{z\}\times\B_{\alpha(z)}:z\in H\}$ is uniformly computable. So $\calH[\{z\}\times\B_{\alpha(z)}:z\in H]$ is a computable copy of $\calH[\A]$, but every isomorphism between these two copies computes $\emptyset'$. So $\calH[(\omega,S)]$ also has strong degree of categoricity $\bfzero'$.
\end{exmp}
However, it appears that the full power of uniform $\bfd$-computable categoricity of $\A$ may not be required for $\calH[\A]$ to be $\bfd$-computably categorical since the following weakening is sufficient:
\begin{defn}
    If $\bfC=\{\calC_i:i<\omega\}$ is a uniformly computable collection of copies of $\A$, we say that $\A$ is \emph{uniformly $\bfd$-computably categorical within $\bfC$} if there is a total $\bfd$-computable binary function $h$ such that for all $i<\omega$, $h(i,\cdot)$ is an isomorphism from $\A$ to $\calC_i$.
\end{defn}
\begin{prop}
    $\calH[\A]$ is $\bfd$-computably categorical if and only if $\A$ is uniformly $\bfd$-computably categorical within $\bfC$ for every uniformly computable collection $\bfC$ of copies of $\A$.
\end{prop}
\begin{proof}
    Suppose $\calH[\A]$ is $\bfd$-computably categorical and let $\bfC=\{\calC_i:i<\omega\}$ be a uniformly computable collection of copies of $\A$. Define $\alpha:H\to\omega$ so that $\alpha(X)=0$ for all $X\in\finsets$, and $\alpha(i,a)=i+1$ for $i<\omega$ and $a<2$. Then $\{\{z\}\times\calC_{\alpha(z)}:z\in H\}$ is a pairwise disjoint uniformly computable collection of copies of $\A$. Since $\calH[\A]$ is $\bfd$-computably categorical, there is a $\bfd$-computable isomorphism $\rho:\calH[\A]\cong\calH[\{z\}\times\calC_{\alpha(z)}:z\in H]$ and there is an $X\in\finsets$ such that $\rho\restrict_{\{\emptyset\}\times A}:\{\emptyset\}\times\A\cong\{X\}\times\calC_0$ and $\rho\restrict_{\{(i,0)\}\times A}:\{(i,0)\}\times\A\cong\{(i,X(i))\}\times\calC_{i+1}$ for each $i<\omega$. Now define a function $h$ so that $h(0):=\pi_2(\rho\restrict_{\{\emptyset\}\times A}(X))$ and $h(i+1):=\pi_2(\rho\restrict_{\{(i,0)\}\times A}(i,X(i)))$ where $\pi_2$ is the projection onto the second coordinate. Then $h$ is total $\bfd$-computable and $h(i)$ is an isomorphism from $\A$ to $\calC_i$ for all $i<\omega$.

    Now suppose that $\A$ is uniformly $\bfd$-computably categorical within $\bfC$ for every uniformly computable collection $\bfC$ of copies of $\A$ and let $\calH[\B_z:z\in H]$ be an arbitrary computable copy of $\calH[\A]$. Fix a computable bijection $\eta:\omega\to H$. Then $\{\B_{\eta(i)}:i<\omega\}$ is uniformly computable, so there is a total $\bfd$-computable function $h$ such that $h(i)$ is an isomorphism from $\A$ to $\B_{\eta(i)}$. Define a function $\rho:\calH[\A]\to\calH[\B_z:z\in H]$ by so that for $z\in H$ and $a\in A$, $\rho(z):=z$ and $\rho(z,a):=h(\eta\inv(z),a)$. Then $\rho\restrict_{\{z\}\times A}$ is an isomorphism from $\{z\}\times\A$ to $\B_{\eta(\eta\inv(z))}=\B_z$. Thus, $\rho$ is a $\bfd$-computable isomorphism from $\calH[\A]$ to $\calH[\B_z:z\in H]$.
\end{proof}

In the previous examples, $\A$ was uniformly $\bfd$-computably categorical exactly when $\A$ was uniformly $\bfd$-computably categorical within every uniformly computable collection of copies. It is not clear whether this is true in general:
\begin{quest}
    Is there a computable structure $\A$ and a Turing degree $\bfd$ such that $\A$ is not uniformly $\bfd$-computably categorical, but is uniformly $\bfd$-computably categorical within every uniformly computable collection of copies of $\A$?
\end{quest}

\pagebreak

\nocite{*}
\bibliographystyle{alpha}
\bibliography{refs}

\end{document}